\documentclass[10pt,reqno,a4paper,oneside,11pt]{amsart}%
\usepackage{amsfonts}
\usepackage{amsfonts}
\usepackage{amsfonts}
\usepackage{amsfonts}
\usepackage{amsfonts}
\usepackage{amsfonts}
\usepackage{amsfonts}
\usepackage{amsfonts}
\usepackage{amsfonts}
\usepackage{amsfonts}
\usepackage{amsfonts}
\usepackage{amsfonts}
\usepackage{amsfonts}
\usepackage{amsfonts}
\usepackage{amsfonts}
\usepackage{amsfonts}
\usepackage{amsfonts}
\usepackage{amsfonts}
\usepackage{amsfonts}
\usepackage{amsfonts}
\usepackage{mathrsfs}
\usepackage{mathrsfs}
\usepackage{amsfonts}
\usepackage{amssymb}
\usepackage{amsmath}
\usepackage{amsthm}
\usepackage{graphicx}%
\providecommand{\U}[1]{\protect\rule{.1in}{.1in}}
\newtheorem{theorem}{Theorem}[section]
\newtheorem{corollary}[theorem]{Corollary}

\newtheorem{remark}{Remark}[section]

\theoremstyle{definition}
\theoremstyle{remark}
\numberwithin{equation}{section}

\ifx\pdfoutput\relax\let\pdfoutput=\undefined\fi
\newcount\msipdfoutput
\ifx\pdfoutput\undefined\else
\ifcase\pdfoutput\else
\ifx\paperwidth\undefined\else
\ifdim\paperheight=0pt\relax\else\pdfpageheight\paperheight\fi
\ifdim\paperwidth=0pt\relax\else\pdfpagewidth\paperwidth\fi
\fi\fi\fi
\begin{document}
\pagestyle{myheadings}

\begin{center}
{\huge \textbf{A simultaneous decomposition of five real quaternion matrices with applications  }}

\bigskip

{\large \textbf{Zhuo-Heng He, Qing-Wen Wang}}

Department of Mathematics, Shanghai University, Shanghai 200444. P.R. China

E-mail: wqw369@yahoo.com

\end{center}

\begin{quotation}
\noindent\textbf{Abstract:}  In this paper, we construct a simultaneous decomposition of   five real quaternion matrices in which
three of  them have the same column numbers, meanwhile three of them have the same row numbers.
Using the simultaneous
matrix decomposition, we derive  the maximal and minimal ranks  of some real quaternion
matrices expressions.  We also show how to choose the variable real quaternion matrices such that the real quaternion matrix expressions
achieve their maximal and minimal ranks. As an application, we give a solvability condition and the general solution to
the real quaternion matrix equation $BXD+CYE=A$.  Moreover,   we give a simultaneous decomposition of seven real quaternion matrices.
\end{quotation}

\section{\textbf{Introduction}}

Let $\mathbb{R}$   be the real   number fields.
Let $\mathbb{H}^{m\times
n}$ be the set of all $m\times n$ matrices over the real quaternion algebra
\[
\mathbb{H}=\big\{a_{0}+a_{1}i+a_{2}j+a_{3}k\big|~i^{2}=j^{2}=k^{2}%
=ijk=-1,a_{0},a_{1},a_{2},a_{3}\in\mathbb{R}\big\}.
\]
For a quaternion matrix $A,$ we denote the conjugate
transpose, the column right space, the row left space of $A$ by $A^{\ast
},\mathcal{R}\left(  A\right)  $, $\mathcal{N}\left(  A\right)  ,$
respectively, the dimension of $\mathcal{R}\left(  A\right)  $ by
$\dim \mathcal{R}\left(  A\right)  .$ By \cite{1}, $\dim \mathcal{R}\left(
A\right)  =\dim$ $\mathcal{N}\left(  A\right)  ,$ which is called the rank of
the quaternion matrix $A$ and denoted by $r(A).$ It is easy to see that for any nonsingular matrices $P$ and $Q$ of appropriate sizes, $A$ and $PAQ$ have the same rank \cite{zhangfuzheng}. Moreover, for $A\in \mathbb{H}^{m\times n},$ by \cite{1}, there exist invertible matrices $P$ and $Q$
such that
\[
PAQ=\left(
\begin{array}{ll}
I_{r} & 0 \\
0 & 0%
\end{array}%
\right)
\]%
where $r=r(A),$ $I_{r}$ is the $r\times r$ identity matrix.

There is no need to emphasize the importance of matrix decomposition. Matrix decomposition is not only
a basic approach  in matrix theory, but also an applicable tool in other
areas in mathematics. Many papers have presented different matrix decompositions for different purposes (e.g. \cite{tree3},
\cite{dlch},  \cite{Y.LIU1}-\cite{ccp}, \cite{Took4}, \cite{ccp0}, \cite{weimusheng}, \cite{F.Zhang3}). For matrix factorization, there are many kinds of decompositions, for instance, the generalized singular decompose (\cite{van}, \cite{ccp}),  the restricted singular value decomposition of
matrix triplets \cite{tree1},  the decomposition of triple matrices which two of them have the same row numbers
meanwhile two of them have the same column numbers (\cite{dlch}, \cite{wiliam}),
the decomposition of  the matrix triplet $(A,B,C)$ (\cite{Y.LIU1}, \cite{Y.LIU244}, \cite{QWWangandyushaowen}),
the decomposition of  a matrix quaternity in which two of
them have the same column numbers, meanwhile three of them have the same row numbers \cite{zhangxia}.

The research  on maximal and minimal ranks of partial matrices
started in later 1980s. Some optimization problems on ranks
of matrix expressions attract much attention from both theoretical and practical points of view.
Minimal and maximal ranks   can be used in control theory (e.g. \cite{CHU1}, \cite{CHU5}). Cohen et al. \cite{cohen1} and  H.J. Woerdeman
\cite{Woerdeman1}-\cite{Woerdeman3} considered the maximal and minimal ranks of $3\times 3$ partial block matrix
\begin{align*}
\begin{bmatrix}A_{11}&A_{12}&X\\A_{21}&A_{22}&A_{23}\\Y&A_{32}&A_{33}\end{bmatrix}
\end{align*}
over $\mathbb{C}$.  Cohen and Dancis \cite{cohen} gave the minimal rank and extremal inertias of the Hermitian matrix
\begin{align*}
\begin{bmatrix}A&B&X\\B^{*}&C&D\\X^{*}&D^{*}&E\end{bmatrix}.
\end{align*}Liu \cite{liuLAMA} obtained the maximal and minimal ranks of $A-BXC$ using the restricted singular value decomposition (RSVD)
of the matrix triplet $(C,A,B)$.  Liu and Tian \cite{liuandtian} derived the maximal and minimal ranks of $A-BX-CY$ using the QQ-SVD
of the matrix triplet $(C,A,B)$.  Tian (\cite{Tian26}, \cite{Tian28}) studied the maximal and minimal ranks of the matrix expression
\begin{align}\label{fun02}
p(X,Y)=A-BXD-CYE
\end{align}
using generalized inverses of matrices. Chu, Hung and Woerdeman \cite{CHU3}  also considered the maximal and minimal ranks of
(\ref{fun02}).

The aim of this paper is to
revisit the maximal and minimal  ranks of the real quaternion matrix expression (\ref{fun02}) through a simultaneous
decomposition of the real quaternion matrix array
\begin{align}\label{fun01}
\begin{matrix}A&B&C\\D&&\\E&&\end{matrix}.
\end{align}
Our method is different from the methods mentioned in \cite{CHU3}, \cite{Tian26} and \cite{Tian28}.
We can derive the maximal and minimal ranks of the real quaternion matrix expression
\begin{align}\label{fun03}
f_{2}(X_{1},X_{2},X_{3},X_{4})=A-B_{1}X_{1}-X_{2}C_{2}-B_{3}X_{3}C_{3}-B_{4}X_{4}C_{4}.
\end{align}
We also consider the  solvability condition and minimal rank of the general solution to the real quaternion matrix equation
\begin{align}\label{sys01}
BXD+CYE=A.
\end{align}
Moreover,   we also give a simultaneous decomposition of the real quaternion matrix array
\begin{align*}
\begin{matrix}A&B&C&D\\E&&&\\F&&&\\G&&&\end{matrix},
\end{align*}
which  plays an important role in investigating the extreme ranks of the matrix expression $A-BXE-CYF-DZG$.

\section{\textbf{A simultaneous decomposition of the  real quaternion matrix array (\ref{fun01})}}

Now we give the main theorem of this section.

\begin{theorem}
\label{theorem01}Let $A\in \mathbb{H}^{m\times n},
B\in \mathbb{H}^{m\times p_{1}},C\in \mathbb{H}^{m\times p_{2}},D\in \mathbb{H}^{q_{1}\times n}$ and $E\in \mathbb{H}^{q_{2}\times n}$  be given. Then there exist
nonsingular matrices $P\in \mathbb{H}^{m\times m}$, $Q\in \mathbb{H}^{n\times n},T_{1}\in \mathbb{H}^{p_{1}\times p_{1}},
T_{2}\in \mathbb{H}^{p_{2}\times p_{2}},V_{1}\in \mathbb{H}^{q_{1}\times q_{1}},V_{2}\in \mathbb{H}^{q_{2}\times q_{2}}$ such that%
\begin{align}\label{equ021}
A=P S_{A}Q,B=P S_{B}T_{1}, C=P S_{C}T_{2}, D=V_{1} S_{D}Q, E=V_{2} S_{E}Q,
\end{align}
where%
\[
S_{A}=\left[
\begin{array}
[c]{ccccccc}%
0 & 0 & 0 & 0 & 0 & I_{m_{1}} & 0\\
0 & A_{1} & A_{2} & A_{3} &   0 & 0 & 0\\
0 & A_{4} & A_{5} & A_{6} & 0 & 0 & 0\\
0 & A_{7} & A_{8} & A_{9} & 0 & 0 & 0\\
0 & 0 & 0 & 0 & I_{m_{5}} & 0 & 0\\
I_{m_{6}} & 0 & 0 & 0 &  0 & 0 & 0 \\
0 & 0 & 0 & 0 &   0 & 0 & 0
\end{array}
\right],
S_{B}=\left[
\begin{array}
[c]{ccc}%
0 & 0&B_{1}\\
I_{m_{2}} &0&0\\
0 & I_{m_{3}}&0\\
0 & 0&0 \\
0 & 0&0 \\
0 & 0&0 \\
0 & 0&0
\end{array}
\right]
,
S_{C}=\left[
\begin{array}
[c]{ccc}%
C_{1} & C_{2}&0\\
0&0&0\\
0&I_{m_{3}} &0\\
0&0 & I_{m_{4}}\\
0 & 0&0 \\
0 & 0&0 \\
0 & 0&0
\end{array}
\right]
,
\]%
\[
S_{D}=\left[
\begin{array}
[c]{ccccccc}%
0 & I_{n_{2}}&0&0&0&0&0\\
0 &0& I_{n_{3}}&0&0&0&0\\
D_{1} & 0&0&0&0&0&0
\end{array}
\right]
,
S_{E}=\left[
\begin{array}
[c]{ccccccc}%
E_{1} & 0&0&0&0&0&0\\
E_{2}&0& I_{n_{3}}&0&0&0&0\\
0 & 0&0&I_{n_{4}}&0&0&0
\end{array}
\right]
,
\]
\begin{align}
&m_{1}+m_{5}+m_{6}=r\begin{bmatrix}A&B&C\end{bmatrix}+r\begin{bmatrix}A\\D\\E\end{bmatrix}-r\begin{bmatrix}A&B&C\\D&0&0\\E&0&0\end{bmatrix},\label{equ022}\\
&m_{2}=r\begin{bmatrix}A&B&C\\D&0&0\\E&0&0\end{bmatrix}-r\begin{bmatrix}A&C\\D&0\\E&0\end{bmatrix},
m_{4}=r\begin{bmatrix}A&B&C\\D&0&0\\E&0&0\end{bmatrix}-r\begin{bmatrix}A&B\\D&0\\E&0\end{bmatrix},\\
&m_{3}=r\begin{bmatrix}A&B\\D&0\\E&0\end{bmatrix}+r\begin{bmatrix}A&C\\D&0\\E&0\end{bmatrix}-r\begin{bmatrix}A&B&C\\D&0&0\\E&0&0\end{bmatrix}-
r\begin{bmatrix}A\\D\\E\end{bmatrix},\\
&n_{2}=r\begin{bmatrix}A&B&C\\D&0&0\\E&0&0\end{bmatrix}-r\begin{bmatrix}A&B&C\\ E&0&0\end{bmatrix},
n_{4}=r\begin{bmatrix}A&B&C\\D&0&0\\E&0&0\end{bmatrix}-r\begin{bmatrix}A&B&C\\ D&0&0\end{bmatrix},\\&
n_{3}=r\begin{bmatrix}A&B&C\\D&0&0 \end{bmatrix}+r\begin{bmatrix}A&B&C\\E&0&0 \end{bmatrix}-r\begin{bmatrix}A&B&C\\D&0&0\\E&0&0\end{bmatrix}-
r\begin{bmatrix}A&B&C\end{bmatrix}.\label{equ026}
\end{align}

\end{theorem}

\begin{proof}The proof is inspired by \cite{Y.LIU244} and \cite{QWWangandyushaowen}. The whole procedure is
constructive and consists of seven steps.

Step 1. We can find two nonsingular matrices $P_{1}$ and $Q_{1}$ such that
\begin{align*}
P_{1}\begin{bmatrix}B&C\end{bmatrix}=\begin{bmatrix}B^{(1)}&C^{(1)}\\0&0\end{bmatrix},
\begin{bmatrix}D\\E\end{bmatrix}Q_{1}=\begin{bmatrix}D^{(1)}&0\\E^{(1)}&0\end{bmatrix},
\end{align*}
where $\begin{bmatrix}B^{(1)}&C^{(1)}\end{bmatrix}$ has full row rank, and $\begin{bmatrix}D^{(1)}\\E^{(1)}\end{bmatrix}$
has full column rank. Denote
\begin{align*}
P_{1}AQ_{1}=\begin{bmatrix}A_{1}^{(1)}&A_{2}^{(1)}\\A_{3}^{(1)}&A_{4}^{(1)}\end{bmatrix}.
\end{align*}

Step 2. There exist two nonsingular matrices $P_{2}$ and $Q_{2}$ such that
\begin{align*}
P_{2}A_{4}^{(1)}Q_{2}=\begin{bmatrix}I_{m_{5}}&0\\0&0\end{bmatrix},~m_{5}=r(A_{4}^{(1)}).
\end{align*}
Then
\begin{align*}
diag(I,P_{2})\begin{bmatrix}A_{1}^{(1)}&A_{2}^{(1)}\\A_{3}^{(1)}&A_{4}^{(1)}\end{bmatrix}
diag(I,Q_{2}):=\begin{bmatrix}A_{1}^{(2)}&A_{2}^{(2)}&A_{3}^{(2)}\\A_{4}^{(2)}&I_{m_{5}}&0\\A_{5}^{(2)}&0&0\end{bmatrix},
\end{align*}
\begin{align*}
diag(I,P_{2})\begin{bmatrix}B^{(1)}&C^{(1)}\\0&0\end{bmatrix}:=\begin{bmatrix}B^{(2)}&C^{(2)} \\0&0\\ 0&0\end{bmatrix},
\begin{bmatrix}D^{(1)}&0\\E^{(1)}&0\end{bmatrix}diag(I,Q_{2}):=\begin{bmatrix}D^{(2)}&0&0\\E^{(2)}&0&0\end{bmatrix},
\end{align*}
where $\begin{bmatrix}B^{(2)}&C^{(2)}\end{bmatrix}$ has full row rank, and $\begin{bmatrix}D^{(2)}\\E^{(2)}\end{bmatrix}$
has full column rank.

Step 3. Set
\begin{align*}
P_{3}=\begin{bmatrix}I&-A_{2}^{(2)}&0\\0&I&0\\0&0&I\end{bmatrix},
Q_{3}=\begin{bmatrix}I&0&0\\-A_{4}^{(2)}&I&0\\0&0&I\end{bmatrix}.
\end{align*}
Then we have
\begin{align*}
P_{3}\begin{bmatrix}A_{1}^{(2)}&A_{2}^{(2)}&A_{3}^{(2)}\\A_{4}^{(2)}&I_{m_{5}}&0\\A_{5}^{(2)}&0&0\end{bmatrix}Q_{3}:=
\begin{bmatrix}A_{1}^{(3)}&0&A_{2}^{(3)}\\0&I_{m_{5}}&0\\A_{3}^{(3)}&0&0\end{bmatrix},
\end{align*}
\begin{align*}
P_{3}\begin{bmatrix}B^{(2)}&C^{(2)} \\0&0\\ 0&0\end{bmatrix}:=
\begin{bmatrix}B^{(3)}&C^{(3)} \\0&0\\ 0&0\end{bmatrix},
\begin{bmatrix}D^{(2)}&0&0\\E^{(2)}&0&0\end{bmatrix}Q_{3}:=\begin{bmatrix}D^{(3)}&0&0\\E^{(3)}&0&0\end{bmatrix}.
\end{align*}

Step 4. We can choose nonsingular matrices $P_{4},Q_{4},P_{5}$ and $Q_{5}$ such that such that
\begin{align*}
P_{4}A_{2}^{(3)}Q_{4}=\begin{bmatrix}I_{m_{1}}&0\\0&0\end{bmatrix},~
P_{5}A_{3}^{(3)}Q_{5}=\begin{bmatrix}I_{m_{6}}&0\\0&0\end{bmatrix},~r(A_{2}^{(3)})=m_{1},~r(A_{3}^{(3)})=m_{6}.
\end{align*}
Then
\begin{align*}
diag(P_{4},I,P_{5})\begin{bmatrix}A_{1}^{(3)}&0&A_{2}^{(3)}\\0&I_{m_{5}}&0\\A_{3}^{(3)}&0&0\end{bmatrix}
diag(Q_{4},I,Q_{5}):=\begin{bmatrix}A_{1}^{(4)}&A_{2}^{(4)}&0&I_{m_{1}}&0\\
A_{3}^{(4)}&A_{4}^{(4)}&0&0&0\\0&0&I_{m_{5}}&0&0\\I_{m_{6}}&0&0&0&0\\0&0&0&0&0\end{bmatrix},
\end{align*}
\begin{align*}
diag(P_{4},I,P_{5})\begin{bmatrix}B^{(3)}&C^{(3)} \\0&0\\ 0&0\end{bmatrix}:=\begin{bmatrix}B_{1}^{(4)}&C_{1}^{(4)} \\
B_{2}^{(4)}&C_{2}^{(4)} \\0&0\\0&0\\  0&0\end{bmatrix},
\end{align*}
\begin{align*}
\begin{bmatrix}D^{(3)}&0&0\\E^{(3)}&0&0\end{bmatrix}
diag(Q_{4},I,Q_{5}):=\begin{bmatrix}D_{1}^{(4)}&D_{2}^{(4)}&0&0&0\\E_{1}^{(4)}&E_{2}^{(4)}&0&0&0\end{bmatrix}.
\end{align*}
Especially, we have
\begin{align*}
r\begin{bmatrix}B&C\end{bmatrix}=r\begin{bmatrix}B^{(3)}&C^{(3)}  \end{bmatrix}=r\begin{bmatrix}B_{1}^{(4)}&C_{1}^{(4)} \\
B_{2}^{(4)}&C_{2}^{(4)}\end{bmatrix},
r\begin{bmatrix}D&E\end{bmatrix}=r\begin{bmatrix}D^{(3)}\\E^{(3)}  \end{bmatrix}=r\begin{bmatrix}D_{1}^{(4)}&D_{2}^{(4)} \\E_{1}^{(4)}&E_{2}^{(4)} \end{bmatrix},
\end{align*}
i.e., $\begin{bmatrix}B_{2}^{(4)}&C_{2}^{(4)}\end{bmatrix}$ has full row rank, and $\begin{bmatrix}D_{2}^{(4)}\\E_{2}^{(4)}\end{bmatrix}$
has full column rank.

Step 5. Set
\begin{align*}
P_{6}=\begin{bmatrix}I&0&0&-A_{1}^{(4)}&0\\0&I&0&-A_{3}^{(4)}&0\\0&0&I&0&0\\0&0&0&I&0\\0&0&0&0&I\end{bmatrix},
Q_{6}=\begin{bmatrix}I&0&0&0&0\\0&I&0&0&0\\0&0&I&0&0\\0&-A_{2}^{(4)}&0&I&0\\0&0&0&0&I\end{bmatrix}.
\end{align*}
Then we have
\begin{align*}
P_{6}\begin{bmatrix}A_{1}^{(4)}&A_{2}^{(4)}&0&I_{m_{1}}&0\\
A_{3}^{(4)}&A_{4}^{(4)}&0&0&0\\0&0&I_{m_{5}}&0&0\\I_{m_{6}}&0&0&0&0\\0&0&0&0&0\end{bmatrix}Q_{6}:=
\begin{bmatrix}0&0&0&I_{m_{1}}&0\\
0&A_{1}^{(5)}&0&0&0\\0&0&I_{m_{5}}&0&0\\I_{m_{6}}&0&0&0&0\\0&0&0&0&0\end{bmatrix},
\end{align*}
\begin{align*}
P_{6}\begin{bmatrix}B_{1}^{(4)}&C_{1}^{(4)} \\
B_{2}^{(4)}&C_{2}^{(4)} \\0&0\\0&0\\  0&0\end{bmatrix}:=
\begin{bmatrix}B_{1}^{(5)}&C_{1}^{(5)} \\
B_{2}^{(5)}&C_{2}^{(5)} \\0&0\\0&0\\  0&0\end{bmatrix},
\begin{bmatrix}D_{1}^{(4)}&D_{2}^{(4)}&0&0&0\\E_{1}^{(4)}&E_{2}^{(4)}&0&0&0\end{bmatrix}Q_{6}:
=\begin{bmatrix}D_{1}^{(5)}&D_{2}^{(5)}&0&0&0\\E_{1}^{(5)}&E_{2}^{(5)}&0&0&0\end{bmatrix},
\end{align*}
where $\begin{bmatrix}B_{2}^{(5)}&C_{2}^{(5)}\end{bmatrix}$ has full row rank, and $\begin{bmatrix}D_{2}^{(5)}\\E_{2}^{(5)}\end{bmatrix}$
has full column rank.

Step 6. We can find six nonsingular matrices $P_{7},Q_{7},W_{B},W_{C},W_{D},W_{E}$ such that
\begin{align*}
P_{7}\begin{bmatrix}B_{2}^{(5)}&C_{2}^{(5)}\end{bmatrix}\begin{bmatrix}W_{B}&0\\0&W_{C}\end{bmatrix}=
\begin{bmatrix}I_{m_{2}}&0&0&0&0&0\\0&I_{m_{3}}&0&0&I_{m_{3}}&0\\0&0&0&0&0&I_{m_{4}}\end{bmatrix},
\end{align*}
\begin{align*}
\begin{bmatrix}W_{D}&0\\0&W_{E}\end{bmatrix}\begin{bmatrix}D_{2}^{(5)}\\E_{2}^{(5)}\end{bmatrix}Q_{7}=
\begin{bmatrix}I_{n_{2}}&0&0\\0&I_{n_{3}}&0 \\0&0&0\\0&0&0\\0&I_{n_{3}}&0\\0&0&I_{n_{4}}\end{bmatrix}.
\end{align*}
Then
\begin{align*}
&diag(I,P_{7},I,I,I)\begin{bmatrix}0&0&0&I_{m_{1}}&0\\
0&A_{1}^{(5)}&0&0&0\\0&0&I_{m_{5}}&0&0\\I_{m_{6}}&0&0&0&0\\0&0&0&0&0\end{bmatrix}diag(I,Q_{7},I,I,I):=\\&
\left[
\begin{array}
[c]{ccccccc}%
0 & 0 & 0 & 0 & 0 & I_{m_{1}} & 0\\
0 & A_{1}^{(6)} & A_{2}^{(6)} & A_{3}^{(6)} &   0 & 0 & 0\\
0 & A_{4}^{(6)} & A_{5}^{(6)} & A_{6}^{(6)} & 0 & 0 & 0\\
0 & A_{7}^{(6)} & A_{8}^{(6)} & A_{9}^{(6)} & 0 & 0 & 0\\
0 & 0 & 0 & 0 & I_{m_{5}} & 0 & 0\\
I_{m_{6}} & 0 & 0 & 0 &  0 & 0 & 0 \\
0 & 0 & 0 & 0 &   0 & 0 & 0
\end{array}
\right],
\end{align*}
\begin{align*}
diag(I,P_{7},I,I,I)\begin{bmatrix}B_{1}^{(5)}&C_{1}^{(5)} \\
B_{2}^{(5)}&C_{2}^{(5)} \\0&0\\0&0\\  0&0\end{bmatrix}\begin{bmatrix}W_{B}&0\\0&W_{C}\end{bmatrix}:=
\begin{bmatrix}B_{1}^{(6)}&B_{2}^{(6)}&B_{3}^{(6)}&C_{1}^{(6)}&C_{2}^{(6)}&C_{3}^{(6)}\\I_{m_{2}}&0&0&0&0&0\\0&I_{m_{3}}&0&0&I_{m_{3}}&0\\0&0&0&0&0&I_{m_{4}}\\
0&0&0&0&0&0\\0&0&0&0&0&0\\0&0&0&0&0&0\end{bmatrix},
\end{align*}
\begin{align*}
\begin{bmatrix}W_{D}&0\\0&W_{E}\end{bmatrix}\begin{bmatrix}D_{1}^{(5)}&D_{2}^{(5)}&0&0&0\\E_{1}^{(5)}&E_{2}^{(5)}&0&0&0\end{bmatrix}diag(I,Q_{7},I,I,I):
=\begin{bmatrix}D_{1}^{(6)}&I_{n_{2}}&0&0&0&0&0\\D_{2}^{(6)}&0&I_{n_{3}}&0 &0&0&0\\D_{3}^{(6)}&0&0&0&0&0&0\\E_{1}^{(6)}&0&0&0&0&0&0\\E_{2}^{(6)}&0&I_{n_{3}}&0&0&0&0\\E_{3}^{(6)}&0&0&I_{n_{4}}&0&0&0\end{bmatrix}.
\end{align*}

Step 7. Set
\begin{align*}
P_{8}=\begin{bmatrix}I&-B_{1}^{(6)}&-B_{2}^{(6)}&-C_{3}^{(6)}&0&0&0\\
0&I&0&0&0&0&0\\
0&0&I&0&0&0&0\\
0&0&0&I&0&0&0\\
0&0&0&0&I&0&0\\
0&0&0&0&0&I&0\\
0&0&0&0&0&0&I
\end{bmatrix},~Q_{8}=\begin{bmatrix}I&0&0&0&0&0&0\\
-D_{1}^{(6)}&I&0&0&0&0&0\\
-D_{2}^{(6)}&0&I&0&0&0&0\\
-E_{3}^{(6)}&0&0&I&0&0&0\\
0&0&0&0&I&0&0\\
0&0&0&0&0&I&0\\
0&0&0&0&0&0&I
\end{bmatrix}.
\end{align*}
Then we have
\begin{align*}
P_{8}\begin{bmatrix}B_{1}^{(6)}&B_{2}^{(6)}&B_{3}^{(6)}&C_{1}^{(6)}&C_{2}^{(6)}&C_{3}^{(6)}\\I_{m_{2}}&0&0&0&0&0\\0&I_{m_{3}}&0&0&I_{m_{3}}&0\\0&0&0&0&0&I_{m_{4}}\\
0&0&0&0&0&0\\0&0&0&0&0&0\\0&0&0&0&0&0\end{bmatrix}:=
\begin{bmatrix}0&0&B_{1} &C_{1} &C_{2} &0\\I_{m_{2}}&0&0&0&0&0\\0&I_{m_{3}}&0&0&I_{m_{3}}&0\\0&0&0&0&0&I_{m_{4}}\\
0&0&0&0&0&0\\0&0&0&0&0&0\\0&0&0&0&0&0\end{bmatrix},
\end{align*}
\begin{align*}
\begin{bmatrix}D_{1}^{(6)}&I_{n_{2}}&0&0&0&0&0\\D_{2}^{(6)}&0&I_{n_{3}}&0 &0&0&0\\D_{3}^{(6)}&0&0&0&0&0&0\\E_{1}^{(6)}&0&0&0&0&0&0\\E_{2}^{(6)}&0&I_{n_{3}}&0&0&0&0\\E_{3}^{(6)}&0&0&I_{n_{4}}&0&0&0\end{bmatrix}
Q_{8}
:=\begin{bmatrix}0&I_{n_{2}}&0&0&0&0&0\\0&0&I_{n_{3}}&0 &0&0&0\\D_{1} &0&0&0&0&0&0\\E_{1} &0&0&0&0&0&0\\E_{2} &0&I_{n_{3}}&0&0&0&0\\0&0&0&I_{n_{4}}&0&0&0\end{bmatrix},
\end{align*}
\begin{align*}
P_{8}\left[
\begin{array}
[c]{ccccccc}%
0 & 0 & 0 & 0 & 0 & I_{m_{1}} & 0\\
0 & A_{1}^{(6)} & A_{2}^{(6)} & A_{3}^{(6)} &   0 & 0 & 0\\
0 & A_{4}^{(6)} & A_{5}^{(6)} & A_{6}^{(6)} & 0 & 0 & 0\\
0 & A_{7}^{(6)} & A_{8}^{(6)} & A_{9}^{(6)} & 0 & 0 & 0\\
0 & 0 & 0 & 0 & I_{m_{5}} & 0 & 0\\
I_{m_{6}} & 0 & 0 & 0 &  0 & 0 & 0 \\
0 & 0 & 0 & 0 &   0 & 0 & 0
\end{array}
\right]Q_{8}:= \left[
\begin{array}
[c]{ccccccc}%
\phi_{1} &\phi_{2} & \phi_{3} & \phi_{4} & 0 & I_{m_{1}} & 0\\
\phi_{5} & A_{1}^{(6)} & A_{2}^{(6)} & A_{3}^{(6)} &   0 & 0 & 0\\
\phi_{6} & A_{4}^{(6)} & A_{5}^{(6)} & A_{6}^{(6)} & 0 & 0 & 0\\
\phi_{7} & A_{7}^{(6)} & A_{8}^{(6)} & A_{9}^{(6)} & 0 & 0 & 0\\
0 & 0 & 0 & 0 & I_{m_{5}} & 0 & 0\\
I_{m_{6}} & 0 & 0 & 0 &  0 & 0 & 0 \\
0 & 0 & 0 & 0 &   0 & 0 & 0
\end{array}
\right],
\end{align*}
where
\begin{align*}
\phi_{1}=-\phi_{2}D_{1}^{(6)}-\phi_{3}D_{2}^{(6)}-\phi_{4}E_{3}^{(6)},
\end{align*}
\begin{align*}
\phi_{2}=-B_{1}^{(6)}A_{1}^{(6)}-B_{2}^{(6)}A_{4}^{(6)}-C_{3}^{(6)}A_{7}^{(6)},~
\phi_{3}=-B_{1}^{(6)}A_{2}^{(6)}-B_{2}^{(6)}A_{5}^{(6)}-C_{3}^{(6)}A_{8}^{(6)},
\end{align*}
\begin{align*}
\phi_{4}=-B_{1}^{(6)}A_{3}^{(6)}-B_{2}^{(6)}A_{6}^{(6)}-C_{3}^{(6)}A_{9}^{(6)},~
\phi_{5}=-A_{1}^{(6)}D_{1}^{(6)}-A_{2}^{(6)}D_{2}^{(6)}-A_{3}^{(6)}E_{3}^{(6)},
\end{align*}
\begin{align*}
\phi_{6}=-A_{4}^{(6)}D_{1}^{(6)}-A_{5}^{(6)}D_{2}^{(6)}-A_{6}^{(6)}E_{3}^{(6)},~
\phi_{7}=-A_{7}^{(6)}D_{1}^{(6)}-A_{8}^{(6)}D_{2}^{(6)}-A_{9}^{(6)}E_{3}^{(6)}.
\end{align*}
Using $I_{m_{1}},I_{m_{6}} $ as the pivots to eliminate the blocks
$\phi_{1},\phi_{2}, \phi_{3}, \phi_{4},\phi_{5},\phi_{6}, \phi_{7} $, respectively, we obtain
\begin{align*}
&\begin{bmatrix}I&0&0&0&0&-\phi_{1}&0\\
0&I&0&0&0&-\phi_{5}&0\\
0&0&I&0&0&-\phi_{6}&0\\
0&0&0&I&0&-\phi_{7}&0\\
0&0&0&0&I&0&0\\
0&0&0&0&0&I&0\\
0&0&0&0&0&0&I
\end{bmatrix}\left[
\begin{array}
[c]{ccccccc}%
\phi_{1} &\phi_{2} & \phi_{3} & \phi_{4} & 0 & I_{m_{1}} & 0\\
\phi_{5} & A_{1}^{(6)} & A_{2}^{(6)} & A_{3}^{(6)} &   0 & 0 & 0\\
\phi_{6} & A_{4}^{(6)} & A_{5}^{(6)} & A_{6}^{(6)} & 0 & 0 & 0\\
\phi_{7} & A_{7}^{(6)} & A_{8}^{(6)} & A_{9}^{(6)} & 0 & 0 & 0\\
0 & 0 & 0 & 0 & I_{m_{5}} & 0 & 0\\
I_{m_{6}} & 0 & 0 & 0 &  0 & 0 & 0 \\
0 & 0 & 0 & 0 &   0 & 0 & 0
\end{array}
\right]\begin{bmatrix}I&0&0&0&0&0&0\\
0&I&0&0&0&0&0\\
0&0&I&0&0&0&0\\
0&0&0&I&0&0&0\\
0&0&0&0&I&0&0\\
0&-\phi_{2}&-\phi_{3}&-\phi_{4}&0&I&0\\
0&0&0&0&0&0&I
\end{bmatrix}\\&:=
\left[
\begin{array}
[c]{ccccccc}%
0 & 0 & 0 & 0 & 0 & I_{m_{1}} & 0\\
0 & A_{1} & A_{2} & A_{3} &   0 & 0 & 0\\
0 & A_{4} & A_{5} & A_{6} & 0 & 0 & 0\\
0 & A_{7} & A_{8} & A_{9} & 0 & 0 & 0\\
0 & 0 & 0 & 0 & I_{m_{5}} & 0 & 0\\
I_{m_{6}} & 0 & 0 & 0 &  0 & 0 & 0 \\
0 & 0 & 0 & 0 &   0 & 0 & 0
\end{array}
\right].
\end{align*}
After Step 7, we have established the decomposition.

\end{proof}
\section{\textbf{Maximal and minimal ranks of (\ref{fun02}) and (\ref{fun03})}}

As applications of  Theorem \ref{theorem01}, we in this section consider the
maximal and minimal ranks of (\ref{fun02}) and (\ref{fun03}) over $\mathbb{H}$. We also
show how to choose the variable real quaternion matrices such that the real quaternion matrix expressions
achieve their maximal and minimal ranks.
\begin{theorem}\label{theorem02}
Let $A,B,C,D,E$ be given quaternion matrices, and $p(X,Y)$ be as given  in (\ref{fun02}). Then
\begin{align*}
 \mathop {\max }\limits_{ X,Y } r\left[ {p\left( {X,Y} \right)} \right]
 = \min \left\{ r\left[ {\begin{array}{*{20}{c}}
   A   \\
  D  \\
   E
\end{array}} \right],r\left[ {\begin{array}{*{20}{c}}
   A & B&C
\end{array}} \right],
 r\left[ {\begin{array}{*{20}{c}}
   A & B  \\
   E & 0
\end{array}} \right],r\left[ {\begin{array}{*{20}{c}}
   A &C \\
   D&0
\end{array}} \right] \right\} ,
\end{align*}
\begin{align*}
\mathop {\min}\limits_{ X,Y } r\left[ {p\left( {X,Y} \right)} \right]
 &= r\left[ {\begin{array}{*{20}{c}}
   A  \\
   D\\E
\end{array}} \right] + r\left[ {\begin{array}{*{20}{c}}
   A & B&C
\end{array}} \right]
  + \max \Bigg\{ r\left[ {\begin{array}{*{20}{c}}
   A & B  \\
  E&0
\end{array}} \right] - r\left[ {\begin{array}{*{20}{c}}
   A & B&C \\
   E&0&0
\end{array}} \right]
\end{align*}
\begin{align*}
- r\left[ {\begin{array}{*{20}{c}}
   A & B  \\
   D&0\\E&0
\end{array}} \right],r\left[ {\begin{array}{*{20}{c}}
   A & C \\
  D&0
\end{array}} \right] - r\left[ {\begin{array}{*{20}{c}}
   A & B&C \\
   D&0&0
\end{array}} \right] - r\left[ {\begin{array}{*{20}{c}}
   A & C \\
   D&0\\E&0
\end{array}} \right]\Bigg\} .
\end{align*}
\end{theorem}

\begin{proof}From Theorem \ref{theorem01}, we can rewrite the expression $p(X,Y)$ in (\ref{fun02}) in the following canonical
form
\begin{align*}
p(X,Y)=P(S_{A}-S_{B}T_{1}XV_{1}S_{D}-S_{C}T_{2}YV_{2}S_{E})Q.
\end{align*}
Because $P,Q$ are nonsingular, the rank of $p(X,Y)$   can be written as
\begin{align*}
r[p(X,Y)]=r(S_{A}-S_{B}T_{1}XV_{1}S_{D}-S_{C}T_{2}YV_{2}S_{E}).
\end{align*}
Put $\widehat{X}=T_{1}XV_{1},\widehat{Y}=T_{2}YV_{2}$. Partition the   matrices $\widehat{X}$ and $\widehat{Y}$
\begin{align*}
\widehat{X}=\begin{bmatrix}X_{1}&X_{2}&X_{3}\\X_{4}&X_{5}&X_{6}\\X_{7}&X_{8}&X_{9}\end{bmatrix},
\widehat{Y}=\begin{bmatrix}Y_{1}&Y_{2}&Y_{3}\\Y_{4}&Y_{5}&Y_{6}\\Y_{7}&Y_{8}&Y_{9}\end{bmatrix}.
\end{align*}
Hence,
\begin{align*}
S_{B}\widehat{X}S_{D}=\left[
\begin{array}
[c]{ccccccc}%
* & * & * & 0 & 0 &0 & 0\\
* & X_{1} & X_{2} & 0 &   0 & 0 & 0\\
* & X_{4} & X_{5} & 0 & 0 & 0 & 0\\
0 & 0 & 0 & 0 & 0 & 0 & 0\\
0 & 0 & 0 & 0 & 0 & 0 & 0\\
0& 0 & 0 & 0 &  0 & 0 & 0 \\
0 & 0 & 0 & 0 &   0 & 0 & 0
\end{array}
\right],
S_{C}\widehat{Y}S_{E}=\left[
\begin{array}
[c]{ccccccc}%
* & 0 & * & * & 0 &0 & 0\\
0 & 0 & 0 & 0 &   0 & 0 & 0\\
* & 0 & Y_{5} & Y_{6} & 0 & 0 & 0\\
* & 0 & Y_{8} & Y_{9} & 0 &0& 0\\
0 & 0 & 0 & 0 & 0 & 0 & 0\\
0& 0 & 0 & 0 &  0 & 0 & 0 \\
0 & 0 & 0 & 0 &   0 & 0 & 0
\end{array}
\right].
\end{align*}
Through the new notation, the matrix expression $A-BXD-CYE$
can be rewritten as
\begin{align*}
S_{A}-S_{B}\widehat{X}S_{D}-S_{C}\widehat{Y}S_{E}=\left[
\begin{array}
[c]{ccccccc}%
* & * & * & * & 0 & I_{m_{1}} & 0\\
* & A_{1}-X_{1} & A_{2}-X_{2} & A_{3} &   0 & 0 & 0\\
* & A_{4}-X_{4} & A_{5}-X_{5}-Y_{5} & A_{6}-Y_{6} & 0 & 0 & 0\\
* & A_{7} & A_{8}-Y_{8} & A_{9}-Y_{9} & 0 & 0 & 0\\
0 & 0 & 0 & 0 & I_{m_{5}} & 0 & 0\\
I_{m_{6}} & 0 & 0 & 0 &  0 & 0 & 0 \\
0 & 0 & 0 & 0 &   0 & 0 & 0
\end{array}
\right].
\end{align*}
Obviously,
\begin{align*}
r(S_{A}-S_{B}\widehat{X}S_{D}-S_{C}\widehat{Y}S_{E})=m_{1}+m_{5}+m_{6}+r(\Omega),
\end{align*}
where
\begin{align}\label{equ035}
\Omega=\left[
\begin{array}
[c]{ccc}%
A_{1}-X_{1} & A_{2}-X_{2} & A_{3} \\
A_{4}-X_{4} & A_{5}-X_{5}-Y_{5}&A_{6}-Y_{6}\\
 A_{7} & A_{8}-Y_{8} & A_{9}-Y_{9}
\end{array}
\right]:=\left[
\begin{array}
[c]{ccc}%
Z_{1} & Z_{2} & A_{3} \\
Z_{3} & Z_{4}&Z_{5}\\
 A_{7} & Z_{6} & Z_{7}
\end{array}
\right],
\end{align}
where $Z_{1},\cdots,Z_{7}$ are arbitrary real quaternion matrices. It is easily seen from (\ref{equ035}) that
\begin{align}\label{equ036}
&\mathop {\max }  \left\{ r(A_{3}),r(A_{7})\right\}\leq r(\Omega)\leq \nonumber\\&
\mathop {\min }  \left\{ m_{2}+m_{3}+m_{4},n_{2}+n_{3}+n_{4},r(A_{3})+n_{2}+n_{3}+m_{3}+m_{4},r(A_{7})+n_{3}+n_{4}+m_{2}+m_{3}\right\}.
\end{align}
Now, we choose $Z_{1},\cdots,Z_{7}$ such that $\Omega$ reach the upper and lower bounds in
 (\ref{equ036}). There exist nonsingular matrices $P_{1},Q_{1},P_{2},Q_{2}$ such that
\begin{align*}
P_{1}A_{3}Q_{1}=\begin{bmatrix}I_{a}&0\\0&0\end{bmatrix},~
P_{2}A_{7}Q_{2}=\begin{bmatrix}I_{b}&0\\0&0\end{bmatrix}.
\end{align*}

Case 1. Assume that $b=r(A_{7})\leq r(A_{3})=a.$ Set
\begin{align*}
Z_{2}=0,Z_{3}=0,Z_{4}=0,Z_{5}=0,Z_{6}=0.
\end{align*}
Then
\begin{align*}
\Omega=\left[
\begin{array}
[c]{cc}%
Z_{1}&A_{3}\\A_{7}&Z_{7}
\end{array}
\right].
\end{align*}
Denote
\begin{align*}
P_{1}Z_{1}Q_{2}=\begin{bmatrix}W_{1}&W_{2}\\W_{3}&W_{4}\\W_{5}&W_{6}\end{bmatrix},~
P_{2}Z_{7}Q_{1}=\begin{bmatrix}U_{1}&U_{2}&U_{3}\\U_{4}&U_{5}&U_{6}\end{bmatrix},
\end{align*}
where $W_{1},\cdots,W_{4},U_{1},\cdots,U_{4}$ are arbitrary real quaternion matrices. Then we have
\begin{align*}
r(\Omega)=r\begin{bmatrix}W_{1}&W_{2}&I_{b}&0&0\\W_{3}&W_{4}&0&I_{a-b}&0\\W_{5}&W_{6}&0&0&0\\
I_{b}&0&U_{1}&U_{2}&U_{3}\\0&0&U_{4}&U_{5}&U_{6}\end{bmatrix}.
\end{align*}
Set
\begin{align*}
W_{1}=I_{b},W_{2}=0,W_{3}=0,W_{4}=0,W_{5}=0,W_{6}=0,\\
U_{1}=I_{b},U_{2}=0,U_{3}=0,U_{4}=0,U_{5}=0,U_{6}=0.
\end{align*}
Then $r(\Omega)=a=r(A_{3})$. The case $r(A_{7})\geq r(A_{3})$ can be shown similarly. Hence,
$
\mathop {\max }  \left\{ r(A_{3}),r(A_{7})\right\}\leq r(\Omega).
$

Case 2. Assume that
\begin{align*}
m_{2}+m_{3}+m_{4} \leq \mathop {\min }  \left\{ n_{2}+n_{3}+n_{4},r(A_{3})+n_{2}+n_{3}+m_{3}+m_{4},r(A_{7})+n_{3}+n_{4}+m_{2}+m_{3}\right\}.
\end{align*}
Then
\begin{align*}
\Omega=
\left[
\begin{array}
[c]{ccccc}%
W_{1}&W_{2}&W_{3}&I_{a}&0\\
W_{4}&W_{5}&W_{6}&0&0\\
W_{7}&W_{8}&W_{9}&W_{10}&W_{11}\\
I_{b}&0&W_{12}&W_{13}&W_{14}\\
0&0&W_{15}&W_{16}&W_{17}
\end{array}
\right].
\end{align*}
where $W_{1},\cdots,W_{17}$ are arbitrary real quaternion matrices. Set
\begin{align*}
&\begin{bmatrix}W_{4}&W_{5}&W_{6}\end{bmatrix}=\begin{bmatrix}I_{m_{4}-b}&0\end{bmatrix},
\begin{bmatrix}W_{11}\\W_{14}\\W_{17}\end{bmatrix}=\begin{bmatrix}I_{m_{3}+b+m_{2}-a}&0\end{bmatrix},
\\&W_{i}=0,i=1,2,3,7,8,9,10,12,13,15,16.
\end{align*}
Then we have
$
r(\Omega)=m_{2}+m_{3}+m_{4}.
$
The case
\begin{align*}
n_{2}+n_{3}+n_{4} \leq \mathop {\min }  \left\{ m_{2}+m_{3}+m_{4},r(A_{3})+n_{2}+n_{3}+m_{3}+m_{4},r(A_{7})+n_{3}+n_{4}+m_{2}+m_{3}\right\}.
\end{align*}
 can be shown similarly.

 Case 3. Assume that
\begin{align*}
r(A_{3})+n_{2}+n_{3}+m_{3}+m_{4}\leq \mathop {\min }  \left\{m_{2}+m_{3}+m_{4}, n_{2}+n_{3}+n_{4},r(A_{7})+n_{3}+n_{4}+m_{2}+m_{3}\right\}.
\end{align*}
Then we have $m_{2}-a\geq n_{2}+n_{3}$ and $n_{4}-a\geq m_{3}+m_{4}$.  Set
\begin{align*}
\begin{bmatrix}W_{4}&W_{5}&W_{6}\end{bmatrix}=\begin{bmatrix}I_{n_{2}+n_{3}}\\0\end{bmatrix},
\begin{bmatrix}W_{11}\\W_{14}\\W_{17}\end{bmatrix}=\begin{bmatrix}I_{m_{3}+m_{4}}&0\end{bmatrix},
W_{i}=0,i=1,2,3,7,8,9,10,12,13,15,16.
\end{align*}
Then we have
\begin{align*}
r(\Omega)=a+n_{2}+n_{3}+m_{3}+m_{4}=r(A_{3})+n_{2}+n_{3}+m_{3}+m_{4}.
\end{align*}
The case
\begin{align*}
r(A_{7})+n_{3}+n_{4}+m_{2}+m_{3}\leq \mathop {\min }  \left\{m_{2}+m_{3}+m_{4}, n_{2}+n_{3}+n_{4},r(A_{3})+n_{2}+n_{3}+m_{3}+m_{4}\right\}.
\end{align*}
 can be shown similarly.

Hence,
\begin{align*}
&m_{1}+m_{5}+m_{6}+\mathop {\max }  \left\{ r(A_{3}),r(A_{7})\right\}\leq r(A-BXD-CYE)\leq m_{1}+m_{5}+m_{6}+ \nonumber\\&
\mathop {\min }  \left\{ m_{2}+m_{3}+m_{4},n_{2}+n_{3}+n_{4},r(A_{3})+n_{2}+n_{3}+m_{3}+m_{4},r(A_{7})+n_{3}+n_{4}+m_{2}+m_{3}\right\}.
\end{align*}
It follows from Theorem \ref{theorem01} that
\begin{align}\label{equ0033}
m_{1}+m_{2}+m_{3}+r(A_{7})+m_{5}+m_{6}+n_{3}+n_{4}=r\begin{bmatrix}A&B\\E&0\end{bmatrix},
\end{align}
\begin{align}\label{equ0034}
m_{1}+n_{2}+n_{3}+m_{3}+m_{4}+r(A_{3})+m_{5}+m_{6}=r\begin{bmatrix}A&C\\D&0\end{bmatrix}.
\end{align}
Combining (\ref{equ022})-(\ref{equ026}), (\ref{equ0033}) and (\ref{equ0034}), we can obtain the results.

\end{proof}

\begin{corollary}\label{coro02}
Let $A,B,C,D,E$ be given, and $p(X,Y)$ be as given in (\ref{fun02}). Assume that
\begin{align*}
\mathcal{R}(B)\subseteq \mathcal{R}(C),~\mathcal{R}(E^{*})\subseteq \mathcal{R}(D^{*}).
\end{align*}
Then
\begin{align*}
 \mathop {\max }\limits_{ X,Y } r\left[ {p\left( {X,Y} \right)} \right]
 = \min \left\{ r\left[ {\begin{array}{*{20}{c}}
   A   \\
  D
\end{array}} \right],r\left[ {\begin{array}{*{20}{c}}
   A &C
\end{array}} \right],
 r\left[ {\begin{array}{*{20}{c}}
   A & B  \\
   E & 0
\end{array}} \right]
\right\} ,
\end{align*}
\begin{align*}
\mathop {\min}\limits_{ X,Y } r\left[ {p\left( {X,Y} \right)} \right]
 &= r\left[ {\begin{array}{*{20}{c}}
   A  \\
   D
\end{array}} \right] + r\left[ {\begin{array}{*{20}{c}}
   A &C
\end{array}} \right]
  + r\left[ {\begin{array}{*{20}{c}}
   A & B  \\
  E&0
\end{array}} \right] - r\left[ {\begin{array}{*{20}{c}}
   A &C \\
   E&0
\end{array}}\right] - r\left[ {\begin{array}{*{20}{c}}
   A & B  \\
   D&0
\end{array}} \right].
\end{align*}
By the method mentioned  in the proof of Theorem \ref{theorem02}, We can choose variable real quaternion matrices $X$ and $Y$
such that the real quaternion matrix expression  attains its maximal and minimal ranks.
\end{corollary}

\begin{corollary}\label{coro01}
Let $A\in \mathbb{H}^{m\times n},B\in \mathbb{H}^{m\times k},C\in \mathbb{H}^{l\times n}$ be given, and
\begin{align}\label{fun04}
f_{1}(X,Y)=A-BX -YC.
\end{align}
Then
\begin{align*}
 \mathop {\max }\limits_{ X,Y } r\left[ {f_{1}\left( {X,Y} \right)} \right]
 = \min \left\{ m,n,
 r\left[ {\begin{array}{*{20}{c}}
   A & B  \\
   C & 0
\end{array}} \right]  \right\} ,
\mathop {\min}\limits_{ X,Y } r\left[ {f_{1}\left( {X,Y} \right)} \right]
 = r\left[ {\begin{array}{*{20}{c}}
   A &B \\
   C&0
\end{array}} \right] -r(B)-r(C).
\end{align*}
By the method mentioned  in the proof of Theorem \ref{theorem02}, We can choose variable real quaternion matrices $X$ and $Y$
such that the real quaternion matrix expression (\ref{fun04}) attains its maximal and minimal ranks.

\end{corollary}

\begin{theorem}\label{theorem04}
Let $A\in \mathbb{H}^{m\times n},B_{1}\in \mathbb{H}^{m\times k},C_{2}\in \mathbb{H}^{l\times n},
B_{3}\in \mathbb{H}^{m\times k_{3}},C_{3}\in \mathbb{H}^{l_{3}\times n},
B_{4}\in \mathbb{H}^{m\times k_{4}},C_{4}\in \mathbb{H}^{l_{4}\times n}$ be given, and
$f_{2}(X_{1},X_{2},X_{3},X_{4})$ be as given in (\ref{fun03}).
Then
\begin{align*}
 \mathop {\max }\limits_{\left\{ {{X_i}} \right\}} r\left[ {f_{2}\left( {{X_1},{X_2},{X_3},{X_4}} \right)} \right]
 =& \min \Bigg\{ m,n,r\left[ {\begin{array}{*{20}{c}}
   A & {{B_1}}  \\
   {{C_2}} & 0  \\
   {{C_3}} & 0  \\
   {{C_4}} & 0  \\
\end{array}} \right],r\left[ {\begin{array}{*{20}{c}}
   A & {{B_1}} & {{B_3}} & {{B_4}}  \\
   {{C_2}} & 0 & 0 & 0  \\
\end{array}} \right], \nonumber\\
 &r\left[ {\begin{array}{*{20}{c}}
   A & {{B_1}} & {{B_3}}  \\
   {{C_2}} & 0 & 0  \\
   {{C_4}} & 0 & 0  \\
\end{array}} \right],r\left[ {\begin{array}{*{20}{c}}
   A & {{B_1}} & {{B_4}}  \\
   {{C_2}} & 0 & 0  \\
   {{C_3}} & 0 & 0  \\
\end{array}} \right] \Bigg\} ,
\end{align*}
\begin{align*}
 \mathop {\min }\limits_{\left\{ {{X_i}} \right\}} r\left[ {f_{2}\left( {{X_1},{X_2},{X_3},{X_4}} \right)} \right]
 &= r\left[ {\begin{array}{*{20}{c}}
   A & {{B_1}}  \\
   {{C_2}} & 0  \\
   {{C_3}} & 0  \\
   {{C_4}} & 0  \\
\end{array}} \right] + r\left[ {\begin{array}{*{20}{c}}
   A & {{B_1}} & {{B_3}} & {{B_4}}  \\
   {{C_2}} & 0 & 0 & 0  \\
\end{array}} \right] - r({B_1}) - r({C_2})
\end{align*}
\begin{align*}
  + \max \left\{ r\left[ {\begin{array}{*{20}{c}}
   A & {{B_1}} & {{B_3}}  \\
   {{C_2}} & 0 & 0  \\
   {{C_4}} & 0 & 0  \\
\end{array}} \right] - r\left[ {\begin{array}{*{20}{c}}
   A & {{B_1}} & {{B_3}} & {{B_4}}  \\
   {{C_2}} & 0 & 0 & 0  \\
   {{C_4}} & 0 & 0 & 0  \\
\end{array}} \right] - r\left[ {\begin{array}{*{20}{c}}
   A & {{B_1}} & {{B_3}}  \\
   {{C_2}} & 0 & 0  \\
   {{C_3}} & 0 & 0  \\
   {{C_4}} & 0 & 0  \\
\end{array}} \right], \right.
\end{align*}
\begin{align*}
 \left. r\left[ {\begin{array}{*{20}{c}}
   A & {{B_1}} & {{B_4}}  \\
   {{C_2}} & 0 & 0  \\
   {{C_3}} & 0 & 0  \\
\end{array}} \right] - r\left[ {\begin{array}{*{20}{c}}
   A & {{B_1}} & {{B_3}} & {{B_4}}  \\
   {{C_2}} & 0 & 0 & 0  \\
   {{C_3}} & 0 & 0 & 0  \\
\end{array}} \right] - r\left[ {\begin{array}{*{20}{c}}
   A & {{B_1}} & {{B_4}}  \\
   {{C_2}} & 0 & 0  \\
   {{C_3}} & 0 & 0  \\
   {{C_4}} & 0 & 0  \\
\end{array}} \right]\right\} .
\end{align*}
\end{theorem}

\begin{proof}Our proof is just similar to  the proof in \cite{Tianlaa}.
Note that
\begin{align}\label{equ034}
\begin{bmatrix}
 A-B_{3}X_{3}C_{3}-B_{4}X_{4}C_{4} &B_{1} \\
   C_{2}&0
\end{bmatrix}=\begin{bmatrix}
 A &B_{1} \\
   C_{2}&0
\end{bmatrix}-\begin{bmatrix}B_{3}\\0\end{bmatrix}X_{3}\begin{bmatrix}C_{3}&0\end{bmatrix}
-\begin{bmatrix}B_{4}\\0\end{bmatrix}X_{4}\begin{bmatrix}C_{4}&0\end{bmatrix}.
\end{align}
It follows from Theorem \ref{theorem02} that there exist $\widehat{X_{3}}$, $\widehat{X_{4}}$
$\widetilde{X_{3}}$ and $\widetilde{X_{4}}$ such that the real quaternion matrix   (\ref{equ034}) attains its maximal and minimal ranks, i.e.
\begin{align*}
&r\begin{bmatrix}
 A-B_{3}\widehat{X_{3}}C_{3}-B_{4}\widehat{X_{4}}C_{4} &B_{1} \\
   C_{2}&0
\end{bmatrix}=\mathop {\max }\limits_{\left\{ {{X_3},X_{4}} \right\}} r\begin{bmatrix}
 A-B_{3}X_{3}C_{3}-B_{4}X_{4}C_{4} &B_{1} \\
   C_{2}&0
\end{bmatrix}\\&=\min \Bigg\{ r\left[ {\begin{array}{*{20}{c}}
   A & {{B_1}}  \\
   {{C_2}} & 0  \\
   {{C_3}} & 0  \\
   {{C_4}} & 0  \\
\end{array}} \right],r\left[ {\begin{array}{*{20}{c}}
   A & {{B_1}} & {{B_3}} & {{B_4}}  \\
   {{C_2}} & 0 & 0 & 0  \\
\end{array}} \right], r\left[ {\begin{array}{*{20}{c}}
   A & {{B_1}} & {{B_3}}  \\
   {{C_2}} & 0 & 0  \\
   {{C_4}} & 0 & 0  \\
\end{array}} \right],r\left[ {\begin{array}{*{20}{c}}
   A & {{B_1}} & {{B_4}}  \\
   {{C_2}} & 0 & 0  \\
   {{C_3}} & 0 & 0  \\
\end{array}} \right] \Bigg\},
\end{align*}
\begin{align*}
&r\begin{bmatrix}
 A-B_{3}\widetilde{X_{3}}C_{3}-B_{4}\widetilde{X_{4}}C_{4} &B_{1} \\
   C_{2}&0
\end{bmatrix}=\mathop {\min }\limits_{\left\{ {{X_3},X_{4}} \right\}} r\begin{bmatrix}
 A-B_{3}X_{3}C_{3}-B_{4}X_{4}C_{4} &B_{1} \\
   C_{2}&0
\end{bmatrix}\\&=r\left[ {\begin{array}{*{20}{c}}
   A & {{B_1}}  \\
   {{C_2}} & 0  \\
   {{C_3}} & 0  \\
   {{C_4}} & 0  \\
\end{array}} \right] + r\left[ {\begin{array}{*{20}{c}}
   A & {{B_1}} & {{B_3}} & {{B_4}}  \\
   {{C_2}} & 0 & 0 & 0  \\
\end{array}} \right]
\end{align*}
\begin{align*}
  + \max \left\{ r\left[ {\begin{array}{*{20}{c}}
   A & {{B_1}} & {{B_3}}  \\
   {{C_2}} & 0 & 0  \\
   {{C_4}} & 0 & 0  \\
\end{array}} \right] - r\left[ {\begin{array}{*{20}{c}}
   A & {{B_1}} & {{B_3}} & {{B_4}}  \\
   {{C_2}} & 0 & 0 & 0  \\
   {{C_4}} & 0 & 0 & 0  \\
\end{array}} \right] - r\left[ {\begin{array}{*{20}{c}}
   A & {{B_1}} & {{B_3}}  \\
   {{C_2}} & 0 & 0  \\
   {{C_3}} & 0 & 0  \\
   {{C_4}} & 0 & 0  \\
\end{array}} \right], \right.
\end{align*}
\begin{align*}
 \left. r\left[ {\begin{array}{*{20}{c}}
   A & {{B_1}} & {{B_4}}  \\
   {{C_2}} & 0 & 0  \\
   {{C_3}} & 0 & 0  \\
\end{array}} \right] - r\left[ {\begin{array}{*{20}{c}}
   A & {{B_1}} & {{B_3}} & {{B_4}}  \\
   {{C_2}} & 0 & 0 & 0  \\
   {{C_3}} & 0 & 0 & 0  \\
\end{array}} \right] - r\left[ {\begin{array}{*{20}{c}}
   A & {{B_1}} & {{B_4}}  \\
   {{C_2}} & 0 & 0  \\
   {{C_3}} & 0 & 0  \\
   {{C_4}} & 0 & 0  \\
\end{array}} \right]\right\} .
\end{align*}
By  Corollary \ref{coro01}, we can find $\widehat{X_{1}}$,
$\widehat{X_{2}}$, $\widetilde{X_{1}}$ and
$\widetilde{X_{2}}$ such that
\begin{align*}
&r\left[ {f_{1}\left( {\widehat{X_{1}},\widehat{X_{2}},\widehat{X_{3}},\widehat{X_{4}}} \right)} \right]=
\min \left\{ m,n,
 r\left[ {\begin{array}{*{20}{c}}
   A-B_{3}\widehat{X_{3}}C_{3}-B_{4}\widehat{X_{4}}C_{4} & B_{1}  \\
   C_{2} & 0
\end{array}} \right]  \right\}\\&=
\min \left\{ m,n,r\left[ {\begin{array}{*{20}{c}}
   A & {{B_1}}  \\
   {{C_2}} & 0  \\
   {{C_3}} & 0  \\
   {{C_4}} & 0  \\
\end{array}} \right],r\left[ {\begin{array}{*{20}{c}}
   A & {{B_1}} & {{B_3}} & {{B_4}}  \\
   {{C_2}} & 0 & 0 & 0  \\
\end{array}} \right],
 r\left[ {\begin{array}{*{20}{c}}
   A & {{B_1}} & {{B_3}}  \\
   {{C_2}} & 0 & 0  \\
   {{C_4}} & 0 & 0  \\
\end{array}} \right],r\left[ {\begin{array}{*{20}{c}}
   A & {{B_1}} & {{B_4}}  \\
   {{C_2}} & 0 & 0  \\
   {{C_3}} & 0 & 0  \\
\end{array}} \right] \right\},
\end{align*}
\begin{align*}
&r\left[ {f_{1}\left( {\widetilde{X_{1}},\widetilde{X_{2}},\widetilde{X_{3}},\widetilde{X_{4}}} \right)} \right]=
r\left[ {\begin{array}{*{20}{c}}
   A-B_{3}\widetilde{X_{3}}C_{3}-B_{4}\widetilde{X_{4}}C_{4} &B_{1} \\
   C_{2}&0
\end{array}} \right] -r(B_{1})-r(C_{2})\\&= r\left[ {\begin{array}{*{20}{c}}
   A & {{B_1}}  \\
   {{C_2}} & 0  \\
   {{C_3}} & 0  \\
   {{C_4}} & 0  \\
\end{array}} \right] + r\left[ {\begin{array}{*{20}{c}}
   A & {{B_1}} & {{B_3}} & {{B_4}}  \\
   {{C_2}} & 0 & 0 & 0  \\
\end{array}} \right] - r({B_1}) - r({C_2})
\end{align*}
\begin{align*}
  + \max \left\{ r\left[ {\begin{array}{*{20}{c}}
   A & {{B_1}} & {{B_3}}  \\
   {{C_2}} & 0 & 0  \\
   {{C_4}} & 0 & 0  \\
\end{array}} \right] - r\left[ {\begin{array}{*{20}{c}}
   A & {{B_1}} & {{B_3}} & {{B_4}}  \\
   {{C_2}} & 0 & 0 & 0  \\
   {{C_4}} & 0 & 0 & 0  \\
\end{array}} \right] - r\left[ {\begin{array}{*{20}{c}}
   A & {{B_1}} & {{B_3}}  \\
   {{C_2}} & 0 & 0  \\
   {{C_3}} & 0 & 0  \\
   {{C_4}} & 0 & 0  \\
\end{array}} \right], \right.
\end{align*}
\begin{align*}
 \left. r\left[ {\begin{array}{*{20}{c}}
   A & {{B_1}} & {{B_4}}  \\
   {{C_2}} & 0 & 0  \\
   {{C_3}} & 0 & 0  \\
\end{array}} \right] - r\left[ {\begin{array}{*{20}{c}}
   A & {{B_1}} & {{B_3}} & {{B_4}}  \\
   {{C_2}} & 0 & 0 & 0  \\
   {{C_3}} & 0 & 0 & 0  \\
\end{array}} \right] - r\left[ {\begin{array}{*{20}{c}}
   A & {{B_1}} & {{B_4}}  \\
   {{C_2}} & 0 & 0  \\
   {{C_3}} & 0 & 0  \\
   {{C_4}} & 0 & 0  \\
\end{array}} \right]\right\} .
\end{align*}

\end{proof}

\begin{theorem}
Let
\begin{align}\label{fun13}
f_{3}(X_{1},X_{2},X_{3},X_{4})=A-B_{1}X_{1}C_{1}-B_{2}X_{2}C_{2}-B_{3}X_{3}C_{3}-B_{4}X_{4}C_{4}
\end{align}
be a linear matrix expression with four two-sided terms over $\mathbb{H}$, where
\begin{align*}
\mathcal{R}(B_{i})\subseteq \mathcal{R}(B_{2}),~\mathcal{R}(C_{j}^{*})\subseteq \mathcal{R}(C_{1}^{*}),~i=1,3,4,~j=2,3,4.
\end{align*}
Then,
\begin{align*}
 \mathop {\max }\limits_{\left\{ {{X_i}} \right\}} r\left[ {f_{3}\left( {{X_1},{X_2},{X_3},{X_4}} \right)} \right]
 =& \min \Bigg\{ r\begin{bmatrix}A&B_{2}\end{bmatrix},r\begin{bmatrix}A\\C_{1}\end{bmatrix},r\left[ {\begin{array}{*{20}{c}}
   A & {{B_1}}  \\
   {{C_2}} & 0  \\
   {{C_3}} & 0  \\
   {{C_4}} & 0  \\
\end{array}} \right],r\left[ {\begin{array}{*{20}{c}}
   A & {{B_1}} & {{B_3}} & {{B_4}}  \\
   {{C_2}} & 0 & 0 & 0  \\
\end{array}} \right], \nonumber\\
 &r\left[ {\begin{array}{*{20}{c}}
   A & {{B_1}} & {{B_3}}  \\
   {{C_2}} & 0 & 0  \\
   {{C_4}} & 0 & 0  \\
\end{array}} \right],r\left[ {\begin{array}{*{20}{c}}
   A & {{B_1}} & {{B_4}}  \\
   {{C_2}} & 0 & 0  \\
   {{C_3}} & 0 & 0  \\
\end{array}} \right] \Bigg\} ,
\end{align*}
\begin{align*}
 &\mathop {\min }\limits_{\left\{ {{X_i}} \right\}} r\left[ {f_{3}\left( {{X_1},{X_2},{X_3},{X_4}} \right)} \right]\\
 &= r\left[ {\begin{array}{*{20}{c}}
   A & {{B_1}}  \\
   {{C_2}} & 0  \\
   {{C_3}} & 0  \\
   {{C_4}} & 0  \\
\end{array}} \right] + r\left[ {\begin{array}{*{20}{c}}
   A & {{B_1}} & {{B_3}} & {{B_4}}  \\
   {{C_2}} & 0 & 0 & 0  \\
\end{array}} \right]+r\begin{bmatrix}A\\C_{1}\end{bmatrix}+r\begin{bmatrix}A&B_{2}\end{bmatrix}- r\begin{bmatrix}A&B_{1}\\C_{1}&0\end{bmatrix}-r\begin{bmatrix}A&B_{2}\\C_{2}&0\end{bmatrix}
\end{align*}
\begin{align*}
+\max \left\{ r\left[ {\begin{array}{*{20}{c}}
   A & {{B_1}} & {{B_3}}  \\
   {{C_2}} & 0 & 0  \\
   {{C_4}} & 0 & 0  \\
\end{array}} \right] - r\left[ {\begin{array}{*{20}{c}}
   A & {{B_1}} & {{B_3}} & {{B_4}}  \\
   {{C_2}} & 0 & 0 & 0  \\
   {{C_4}} & 0 & 0 & 0  \\
\end{array}} \right] - r\left[ {\begin{array}{*{20}{c}}
   A & {{B_1}} & {{B_3}}  \\
   {{C_2}} & 0 & 0  \\
   {{C_3}} & 0 & 0  \\
   {{C_4}} & 0 & 0  \\
\end{array}} \right], \right.
\end{align*}
\begin{align*}
 \left. r\left[ {\begin{array}{*{20}{c}}
   A & {{B_1}} & {{B_4}}  \\
   {{C_2}} & 0 & 0  \\
   {{C_3}} & 0 & 0  \\
\end{array}} \right] - r\left[ {\begin{array}{*{20}{c}}
   A & {{B_1}} & {{B_3}} & {{B_4}}  \\
   {{C_2}} & 0 & 0 & 0  \\
   {{C_3}} & 0 & 0 & 0  \\
\end{array}} \right] - r\left[ {\begin{array}{*{20}{c}}
   A & {{B_1}} & {{B_4}}  \\
   {{C_2}} & 0 & 0  \\
   {{C_3}} & 0 & 0  \\
   {{C_4}} & 0 & 0  \\
\end{array}} \right]\right\} .
\end{align*}
\end{theorem}

\begin{proof}
Applying the method mentioned in Theorem \ref{theorem04} and Corollary \ref{coro02}, we can get the results.
\end{proof}

\begin{remark}
Tian \cite{tianjiaocai} derived the maximal and minimal ranks of (\ref{fun02}), (\ref{fun03}), (\ref{fun04}) and (\ref{fun13}) over $\mathbb{C}$.
\end{remark}
\section{\textbf{Minimal rank of the general solution to real quaternion matrix equation (\ref{sys01})}}

We in this section consider a solvability condition and the general solution to  real quaternion matrix equation (\ref{sys01}). Moreover,
we study the minimal rank of  the general solution to real quaternion matrix equation (\ref{sys01}).

\begin{theorem}\label{theorem03}
Let $A,B,C,D,E$ be given. Then the real matrix equation (\ref{sys01}) is consistent if and only if
\begin{align}\label{equ037}
r\begin{bmatrix}A&C&B\end{bmatrix}=r\begin{bmatrix}C&B\end{bmatrix},
r\begin{bmatrix}A\\D\\E\end{bmatrix}=r\begin{bmatrix}D\\E\end{bmatrix},
r\begin{bmatrix}A&B\\E&0\end{bmatrix}=r\begin{bmatrix}0&B\\E&0\end{bmatrix},
r\begin{bmatrix}A&C\\D&0\end{bmatrix}=r\begin{bmatrix}0&C\\D&0\end{bmatrix}.
\end{align}

In this case, the general solution to (\ref{sys01}) can be expressed as
\begin{align*}
X=T_{1}^{-1}\begin{bmatrix}A_{1}&A_{2}&X_{3}\\A_{4}&X_{5}&X_{6}\\X_{7}&X_{8}&X_{9}\end{bmatrix}V_{1}^{-1},
Y=T_{2}^{-1}\begin{bmatrix}Y_{1}&Y_{2}&Y_{3}\\Y_{4}&A_{5}-X_{5}&A_{6}\\Y_{7}&A_{8}&A_{9}\end{bmatrix}V_{2}^{-1},
\end{align*}
where $A_{1},A_{2},A_{4},A_{5},A_{6},A_{8},A_{9},T_{1},V_{1},T_{2},V_{2}$ are defined as in Theorem \ref{theorem01}, and $X_{3},X_{5},X_{6},X_{7},$ $X_{8},X_{9},
Y_{1},Y_{2},Y_{3},Y_{4},Y_{7}$ are arbitrary real quaternion matrices.
\end{theorem}

\begin{remark}
J.K. Baksalary and R. Kala \cite{J.K.B}, A.B. \"{O}zg\"{u}ler \cite{A.B} derived the solvability condition  (\ref{equ037}).
\end{remark}

\begin{theorem}
Let $A\in \mathbb{H}^{m\times n},
B\in \mathbb{H}^{m\times p_{1}},C\in \mathbb{H}^{m\times p_{2}},D\in \mathbb{H}^{q_{1}\times n}$ and $E\in \mathbb{H}^{q_{2}\times n}$  be given. Suppose that real quaternion matrix equation
(\ref{sys01}) is consistent.
Then\newline%
\begin{align*}
  &\mathop {\min }\limits_{BXD+CYE=A } r\left(  {X }
\right) =r\begin{bmatrix}A&C\end{bmatrix}+r\begin{bmatrix}A\\E\end{bmatrix}-r\begin{bmatrix}A&C\\E&0\end{bmatrix},\\&
  \mathop {\min }\limits_{BXD+CYE=A } r\left(  {Y }
\right) =r\begin{bmatrix}A&B\end{bmatrix}+r\begin{bmatrix}A\\D\end{bmatrix}-r\begin{bmatrix}A&B\\D&0\end{bmatrix}.
\end{align*}

\end{theorem}

\begin{proof}
It follows from Theorem \ref{theorem03} that the solution $X$ in (\ref{sys01}) can be expressed as
\begin{align*}
X=T_{1}^{-1}\begin{bmatrix}A_{1}&A_{2}&X_{3}\\A_{4}&X_{5}&X_{6}\\X_{7}&X_{8}&X_{9}\end{bmatrix}V_{1}^{-1},
\end{align*}
where $A_{1},A_{2},A_{4},T_{1},V_{1}$ are defined as in Theorem \ref{theorem01}, and $X_{3},X_{5},X_{6},X_{7},X_{8},X_{9}$ are arbitrary real quaternion matrices. Since
real quaternion matrix equation
(\ref{sys01}) is consistent, i.e., the equations in (\ref{equ037}) hold, we get
$A_{3}=0,A_{7}=0,B_{1}=0,D_{1}=0,$ where $A_{3},A_{7},B_{1}$ and $D_{1}$ are defined as in Theorem \ref{theorem01}. Note that
\begin{align}\label{equ041}
\begin{bmatrix}A_{1}&A_{2}&X_{3}\\A_{4}&X_{5}&X_{6}\\X_{7}&X_{8}&X_{9}\end{bmatrix}=
\begin{bmatrix}A_{1}&A_{2}&X_{3}\\A_{4}&0&0\\X_{7}&0&0\end{bmatrix}+\begin{bmatrix}0&0\\I&0\\0&I\end{bmatrix}
\begin{bmatrix}X_{5}&X_{6}\\ X_{8}&X_{9}\end{bmatrix}\begin{bmatrix}0&I&0\\0&0&I\end{bmatrix}.
\end{align}
First, we consider the minimal rank of $X$. Applying Theorem \ref{theorem02} to (\ref{equ041}) yields
\begin{align*}
  \mathop {\min }\limits_{BXD+CYE=A } r\left(  {X }
\right) =&\mathop {\min }\limits_{X_{3},X_{5},X_{6},X_{7},X_{8},X_{9} }r\begin{bmatrix}A_{1}&A_{2}&X_{3}\\A_{4}&X_{5}&X_{6}\\X_{7}&X_{8}&X_{9}\end{bmatrix}\\=&
\mathop {\min }\limits_{X_{3},X_{7} }\left(r\begin{bmatrix}A_{1}&A_{2}&X_{3}\end{bmatrix}+r\begin{bmatrix}A_{1}\\A_{4}\\X_{7}\end{bmatrix}-r(A_{1})\right)\\=&
r\begin{bmatrix}A_{1}&A_{2} \end{bmatrix}+r\begin{bmatrix}A_{1}\\A_{4} \end{bmatrix}-r(A_{1}). (Set X_{3}=0, X_{7}=0)
\end{align*}
Applying (\ref{equ022})-(\ref{equ026}) to $r\begin{bmatrix}A&C\end{bmatrix}+r\begin{bmatrix}A\\E\end{bmatrix}-r\begin{bmatrix}A&C\\E&0\end{bmatrix}$ yields
\begin{align*}
r\begin{bmatrix}A&C\end{bmatrix}+r\begin{bmatrix}A\\E\end{bmatrix}-r\begin{bmatrix}A&C\\E&0\end{bmatrix}=r\begin{bmatrix}A_{1}&A_{2} \end{bmatrix}+r\begin{bmatrix}A_{1}\\A_{4} \end{bmatrix}-r(A_{1})= \mathop {\min }\limits_{BXD+CYE=A } r\left(  {X }
\right).
\end{align*}

\end{proof}

\begin{remark}
Tian \cite{tianact} gave the minimal rank of the general solution to the  matrix equation (\ref{sys01}). Our method  is different from this in \cite{tianact}.
\end{remark}

\section{\textbf{A simultaneous decomposition of seven real quaternion matrices }}

In this section, we give a simultaneous decomposition of
the real quaternion matrix array
\begin{align}\label{fun01}
\begin{matrix}A&B&C&D\\E&&&\\F&&&\\G&&&\end{matrix}.
\end{align}
The proof will be given in another paper.

\begin{theorem}\label{theorem05}
Let $A\in \mathbb{H}^{m\times n},
B\in \mathbb{H}^{m\times p_{1}},C\in \mathbb{H}^{m\times p_{2}},D\in \mathbb{H}^{m\times p_{3}},E\in \mathbb{H}^{q_{1}\times n},F\in \mathbb{H}^{q_{2}\times n}$  and $G\in \mathbb{H}^{q_{3}\times n}$be given. Then there exist
nonsingular matrices $P\in \mathbb{H}^{m\times m}$, $Q\in \mathbb{H}^{n\times n},T_{1}\in \mathbb{H}^{p_{1}\times p_{1}},
T_{2}\in \mathbb{H}^{p_{2}\times p_{2}},T_{3}\in \mathbb{H}^{p_{3}\times p_{3}},V_{1}\in \mathbb{H}^{q_{1}\times q_{1}},V_{2}\in \mathbb{H}^{q_{2}\times q_{2}},V_{3}\in \mathbb{H}^{q_{3}\times q_{3}}$ such that%
\begin{align}\label{equ021}
A=P S_{A}Q,B=P S_{B}T_{1}, C=P S_{C}T_{2},D=P S_{D}T_{3}, E=V_{1} S_{E}Q, F=V_{2} S_{F}Q,G=V_{3} S_{G}Q,
\end{align}
where%
\[
S_{A}=\left[
\begin{array}
[c]{ccccccccccccc}%
0 & 0 & 0 & 0 & 0 & 0&0&0&0&0&0&I_{p_{1}} & 0\\
0 & A_{11} & A_{12} & A_{13} &A_{14} & A_{15} & A_{16} &A_{17} & A_{18} & A_{19} &   0 & 0 & 0\\
0 & A_{21} & A_{22} & A_{23} &A_{24} & A_{25} & A_{26} &A_{27} & A_{28} & A_{29} &   0 & 0 & 0\\
0 & A_{31} & A_{32} & A_{33} &A_{34} & A_{35} & A_{36} &A_{37} & A_{38} & A_{39} &   0 & 0 & 0\\
0 & A_{41} & A_{42} & A_{43} &A_{44} & A_{45} & A_{46} &A_{47} & A_{48} & A_{49} &   0 & 0 & 0\\
0 & A_{51} & A_{52} & A_{53} &A_{54} & A_{55} & A_{56} &A_{57} & A_{58} & A_{59} &   0 & 0 & 0\\
0 & A_{61} & A_{62} & A_{63} &A_{64} & A_{65} & A_{66} &A_{67} & A_{68} & A_{69} &   0 & 0 & 0\\
0 & A_{71} & A_{72} & A_{73} &A_{74} & A_{75} & A_{76} &A_{77} & A_{78} & A_{79} &   0 & 0 & 0\\
0 & A_{81} & A_{82} & A_{83} &A_{84} & A_{85} & A_{86} &A_{87} & A_{88} & A_{89} &   0 & 0 & 0\\
0 & A_{91} & A_{92} & A_{93} &A_{94} & A_{95} & A_{96} &A_{97} & A_{98} & A_{99} &   0 & 0 & 0\\
0 & 0 & 0 & 0 & 0 & 0 & 0 & 0 &0&0&I_{p_{2}} & 0 & 0\\
I_{p_{3}} & 0 & 0 & 0 &  0 & 0 & 0 & 0 & 0 & 0 &  0 & 0 & 0 \\
0 & 0 & 0 & 0 &   0 & 0 & 0& 0 & 0 & 0 &  0 & 0 & 0
\end{array}
\right],\]
\[
S_{B}=\left[
\begin{array}
[c]{cccccc}%
0 & 0&0 & 0&0&B_{1}\\
I_{m_{1}} &0&0&0&0&0\\
0 & I_{m_{2}}&0&0&0&0\\
0 & 0&I_{m_{3}}&0&0&0\\
0 & 0&0&I_{m_{4}}&0&0\\
0 & 0&0&0&I_{m_{5}}&0\\
0 & 0&0 &0&0&0\\
0 & 0&0 &0&0&0\\
0 & 0&0 &0&0&0\\
0 & 0&0 &0&0&0\\
0 & 0&0 &0&0&0\\
0 & 0&0 &0&0&0\\
0 & 0&0 &0&0&0
\end{array}
\right]
,
S_{C}=\left[
\begin{array}
[c]{cccccc}%
0 & 0&0 & C_{1}&C_{2}&C_{3}\\
0 &0&0&I_{m_{1}}&0&0\\
0 & 0&0&0&I_{m_{2}}&0\\
0 & 0&0 &0&0&0\\
0 & 0&0 &0&0&0\\
0 & 0&0 &0&0&0\\
I_{m_{4}} & 0&0&0&0&0\\
0 & I_{m_{6}}&0&0&0&0\\
0 & 0&I_{m_{7}}&0&0&0\\
0 & 0&0 &0&0&0\\
0 & 0&0 &0&0&0\\
0 & 0&0 &0&0&0\\
0 & 0&0 &0&0&0
\end{array}
\right]
,
\]%
\[
S_{D}=\left[
\begin{array}
[c]{cccccc}%
0 & D_{1}&D_{2} & D_{3}&D_{4}&D_{5}\\
0 &0&0&0 &I_{m_{1}}&0\\
0 & 0&0 &0&0&0\\
0&0&0&I_{m_{3}}&0&0\\
0&I_{m_{4}} & 0&0&0&0\\
0 & 0&0 &0&0&0\\
0&I_{m_{4}} & 0&0&0&0\\
0 & 0&I_{m_{6}}&0&0&0\\
0 & 0&0 &0&0&0\\
I_{m_{8}}&0 & 0&0&0&0\\
0 & 0&0 &0&0&0\\
0 & 0&0 &0&0&0\\
0 & 0&0 &0&0&0
\end{array}
\right]
,
\]
\[
S_{E}=\left[
\begin{array}
[c]{ccccccccccccc}%
0 & I_{n_{1}} & 0 & 0 & 0 & 0&0&0&0&0&0&0 & 0\\
0 & 0&I_{n_{2}}  & 0 & 0 & 0&0&0&0&0&0&0 & 0\\
0 & 0&0&I_{n_{3}} & 0 & 0 & 0 & 0&0&0&0&0&0\\
0 & 0&0&0&I_{n_{4}} & 0 & 0 & 0 & 0&0&0&0&0\\
0 & 0&0&0&0&I_{n_{5}} & 0 & 0 & 0 & 0&0&0&0\\
E_{1} & 0 & 0 & 0 & 0&0&0&0&0&0&0 & 0&0
\end{array}
\right],\]
\[
S_{F}=\left[
\begin{array}
[c]{ccccccccccccc}%
0 & 0 & 0 & 0 & 0 & 0&I_{n_{4}}&0&0&0&0&0 & 0\\
0 & 0 & 0 & 0 & 0 & 0&0&I_{n_{6}}&0&0&0&0 & 0\\
0 & 0 & 0 & 0 & 0 & 0&0&0&I_{n_{7}}&0&0&0 & 0\\
F_{1} & I_{n_{1}}&0&0&0 & 0 & 0 & 0 & 0&0&0&0&0\\
F_{2} &0& I_{n_{2}}&0&0 & 0 & 0 & 0 & 0&0&0&0&0\\
F_{3} & 0 & 0 & 0 & 0&0&0&0&0&0&0 & 0&0
\end{array}
\right],\]
\[
S_{G}=\left[
\begin{array}
[c]{ccccccccccccc}%
0 & 0 & 0 & 0 & 0 & 0&0&0&0&I_{n_{8}}&0&0 & 0\\
G_{1} & 0&0 & 0 & I_{n_{4}} & 0&I_{n_{4}} &0&0&0&0&0 & 0\\
G_{2} & 0&0 & 0 & 0 & 0&0 &I_{n_{6}}&0&0&0&0 & 0\\
G_{3} & 0&0 & I_{n_{3}} & 0 & 0&0 &0&0&0&0&0 & 0\\
G_{4} & I_{n_{1}}&0 & 0 & 0 & 0&0 &0&0&0&0&0 & 0\\
G_{5} & 0 & 0 & 0 & 0&0&0&0&0&0&0 & 0&0
\end{array}
\right].\]
It is easy to derive the numbers of $p_{1},\cdots,p_{3},m_{1},\cdots,m_{8},n_{1},\cdots,n_{8}$.

\end{theorem}

\begin{remark}
Tian \cite{tianact} derived the maximal  rank of the matrix expression
\begin{align}\label{fun15}
k(X,Y,Z)=A-BXE-CYF-DZG.
\end{align}
However, Tian did not give the proof. Applying Theorem \ref{theorem05}, we can find the maximal and minimal ranks
of the matrix expression (\ref{fun15}). Moreover, we can give a solvability condition and the general solution to
the matrix equation
\begin{align}\label{systhem001}
BXE+CYF+DZG=A.
\end{align}
The corresponding results and their applications
will be given in another paper.

\end{remark}
\begin{remark}
Similarly, we can establish a simultaneous decomposition of $[A,B,C,D]$, where $A=A^{*}$. Using the results on
the simultaneous decomposition of $[A,B,C,D]$, we can derive the maximal rank and extremal inertias of the Hermitian matrix expression
\begin{align}\label{fun16}
h(X,Y,Z)=A-BXB^{*}-CYC^{*}-DZD^{*},~X=X^{*},~Y=Y^{*},~Z=Z^{*}.
\end{align}
In addition, we can give a solvability condition and the general solution to
the matrix equation
\begin{align}\label{systhem002}
BXB^{*}+CYC^{*}+DZD^{*}=A,~X=X^{*},~Y=Y^{*},~Z=Z^{*}.
\end{align}
The corresponding results and their applications
will be given in another paper.

\end{remark}

\begin{remark}
  Tian \cite{Tianlaa} derived the maximal and minimal inertias of the Hermitian matrix expression
\begin{align}\label{fun017}
f_{k}(X_{1},X_{2},\cdots,X_{k})=A-B_{1}X_{1}B_{1}^{*}-\cdots-B_{k}X_{k}B_{k}^{*},X_{i}=X_{i}^{*}
\end{align}
using generalized inverses of matrices. Chen, He and Wang \cite{chendi} derived the maximal  rank of the Hermitian matrix expression
(\ref{fun017})
through  a simultaneous decomposition of the matrix $[A,B_{1},\cdots,B_{k}]$,
where $A$ is Hermitian. The approach in this paper is different with this in \cite{chendi}.
\end{remark}
\section{\textbf{Conclusions}}

We have established a simultaneous decomposition of   five real quaternion matrices in which
three of  them have the same column numbers, meanwhile three of them have the same row numbers.
Using the simultaneous
matrix decomposition, we  have presented the maximal and minimal ranks of the real quaternion matrix
expressions (\ref{fun02}) and (\ref{fun03}). We have given
a necessary and sufficient condition for the existence of the general solution to the real quaternion
matrix equation (\ref{sys01}). The expression of the general solution
to (\ref{sys01}) has also been given when it is solvable. Moreover, we have derived the minimal rank of
the general solution to  the real quaternion
matrix equation (\ref{sys01}). The results of this paper may be generalized to an arbitrary division ring (with an involutional anti-automorphism).

\end{document}